\documentclass[11pt]{article}
\usepackage{a4wide}
\usepackage[utf8]{inputenc}

\usepackage{mathrsfs}  
\usepackage{amsthm}
\usepackage{amssymb}
\usepackage{amsmath}
\newtheorem{theorem}{Theorem}[section]
\newtheorem{lemma}[theorem]{Lemma}
\newtheorem{corollary}[theorem]{Corollary}

\newtheorem{claim}{Claim}
\newtheorem{definition}{Definition}

\newtheorem{coloring}{Coloring Step}

\usepackage{mathtools}

\usepackage[dvipsnames]{xcolor}
\usepackage{tikz}
\usepackage{graphicx}
\graphicspath{ {./images/} }

\tikzset{every loop/.style={}}
\tikzstyle{vertex} = [draw, circle, fill, inner sep=1pt]
\tikzstyle{bndvertex} = [draw, rectangle, minimum width=2pt, minimum height=2pt]
\tikzstyle{essvertex} = [bndvertex, fill=red]
\tikzstyle{essedge} = [line width=1pt, red]
\usetikzlibrary{calc}

\usepackage{caption}
\usepackage{subcaption}
\usepackage{enumitem}
\usepackage{xspace}
\usepackage{xparse}
\usepackage{xfrac}
\usepackage{todonotes}
\usepackage{textpos}

\usepackage{hyperref}

\newcommand{\clique}[1]{\ensuremath{\omega(#1)}}

\newcommand{\Oh}[2][]{{\cal O}_{#1}(#2)}

\newcommand{\N}{\mathbb{N}}
\newcommand{\Z}{\mathbb{Z}}
\newcommand{\R}{\mathbb{R}}

\newcommand{\CC}{\mathscr{C}}

\newcommand{\Rc}{\mathcal{R}}
\newcommand{\Cc}{\mathcal{C}}
\newcommand{\Dc}{\mathcal{D}}
\newcommand{\Ec}{\mathcal{E}}

\newcommand{\Lc}{\mathcal{L}}
\newcommand{\Pc}{\mathcal{P}}
\newcommand{\Sc}{\mathcal{S}}

\newcommand{\Hs}{\mathsf{H}}
\newcommand{\Vs}{\mathsf{V}}
\newcommand{\Ms}{\mathsf{M}}

\def\cqedsymbol{\ifmmode$\lrcorner$\else{\unskip\nobreak\hfil
\penalty50\hskip1em\null\nobreak\hfil$\lrcorner$
\parfillskip=0pt\finalhyphendemerits=0\endgraf}\fi} 
\newcommand{\cqed}{\renewcommand{\qed}{\cqedsymbol}}

\renewcommand{\leq}{\leqslant}
\renewcommand{\geq}{\geqslant}

\begin{document}

\title{Graphs of bounded twin-width\\ are quasi-polynomially $\chi$-bounded\thanks{This work is a part of project BOBR that has received funding from the European Research Council (ERC) under the European Union’s Horizon 2020 research and innovation programme (grant agreement No. 948057).}
}
\author{
Micha\l{} Pilipczuk\thanks{Institute of Informatics, University of Warsaw, Poland (\texttt{michal.pilipczuk@mimuw.edu.pl})} \and Marek Soko\l{}owski\thanks{Institute of Informatics, University of Warsaw, Poland (\texttt{marek.sokolowski@mimuw.edu.pl})}}

\maketitle
\thispagestyle{empty}

\begin{textblock}{20}(-1.9, 8.3)
	\includegraphics[width=40px]{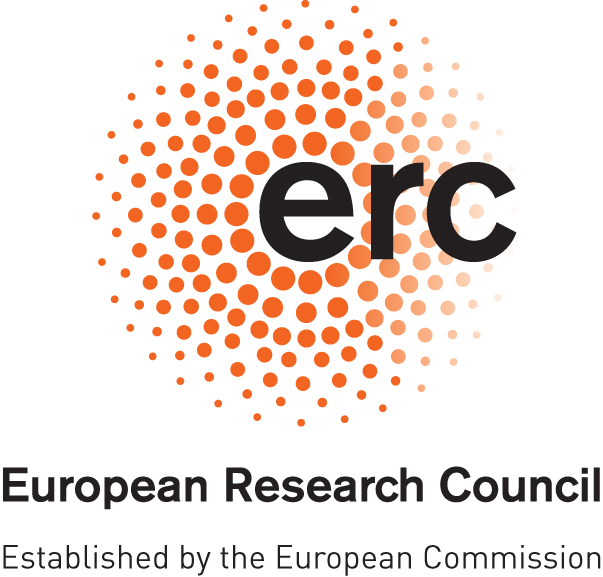}%
\end{textblock}
\begin{textblock}{20}(-2.15, 8.7)
	\includegraphics[width=60px]{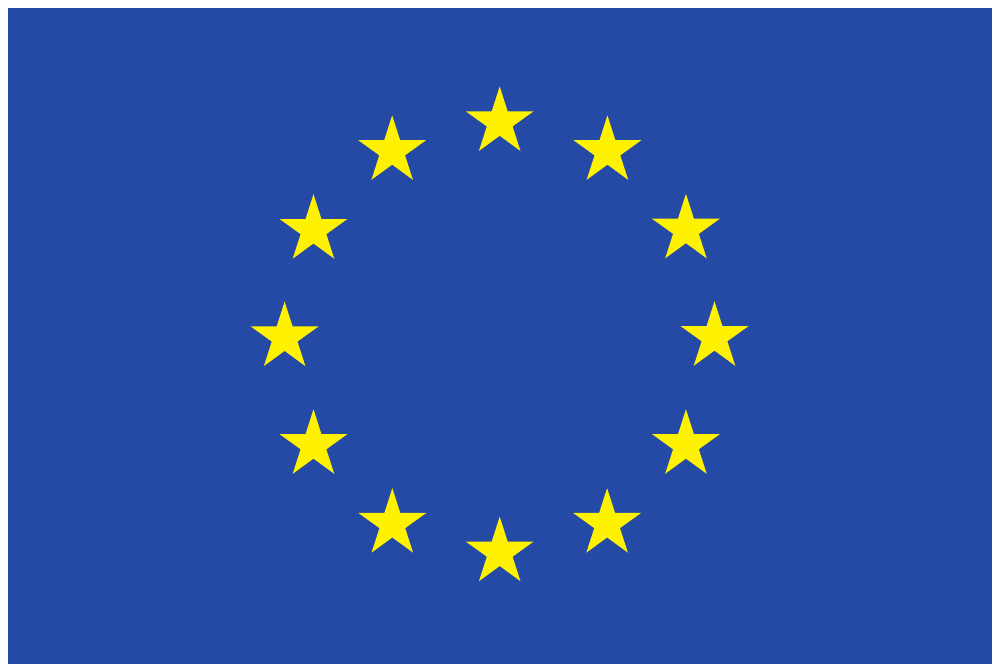}%
\end{textblock}

\begin{abstract}
We prove that for every $t\in \N$ there is a constant $\gamma_t$ such that every graph with twin-width at most $t$ and clique number $\omega$ has chromatic number bounded by $2^{\gamma_t \log^{4t+3} \omega}$. In other words, we prove that graph classes of bounded twin-width are quasi-polynomially \mbox{$\chi$-bounded}. This provides a significant step towards resolving the question of Bonnet et~al.~[ICALP~2021] about whether they are polynomially $\chi$-bounded.

\end{abstract}

\newpage
\pagenumbering{arabic}

\section{Introduction}\label{sec:introduction}

Twin-width is a graph parameter introduced recently by Bonnet et al.~\cite{tww1}. Intuitively, a graph $G$ has twin-width at most $t$ if one can gradually merge the vertices of $G$ into larger and larger subsets, called {\em{parts}}, until all of them are contained in a single part, so that at every point of time, every part has a non-trivial interaction only with at most $t$ other parts. This idea is formalized through the notion of a {\em{contraction sequence}}, see Section~\ref{sec:preliminaries} for details.

Following the work of Bonnet et al.~\cite{tww1}, it very quickly became clear that twin-width is a notion of immense importance that has a prominent place in structural graph theory. There are multiple reasons for this, but perhaps the most convincing is that, as proved in~\cite{tww1}, both graphs of bounded cliquewidth and graphs excluding a fixed minor have bounded twin-width. Thus, twin-width is a robust complexity measure for graphs that is well-suited for the treatment of dense, but well-structured graphs, and at the same time captures graphs that are not necessarily tree-like, for instance grids. Notably, Bonnet et al.~\cite{tww1} proved a suitable grid theorem for twin-width relating the boundedness of twin-width with the non-existence of certain structures in the adjacency matrix. This connection fundamentally relies on the proof of the Stanley-Wilf Conjecture due to Marcus and Tardos~\cite{MarcusTardos}, and is the base for multiple advances in the theory of twin-width, including this work.

Twin-width also appears to be an important concept in finite model theory and its algorithmic aspects, see~\cite{tww4,tww1,BonnetNOdMST21,GajarskyPT21} for a broader discussion. The combinatorics of twin-width were further explored in~\cite{AhnHKO21,BalabanH21,tww2,tww3,tww6,DreierGJOdMR21,JacobP22}, while algorithmic questions around twin-width are considered also in~\cite{BergeBD21,tww3,tww4,BonnetKRTW21,PilipczukSZ21}.

In this work we study a concrete question concerning the combinatorics of twin-width. Recall that a hereditary class of graphs $\CC$ is {\em{$\chi$-bounded}} if there is a function $f\colon \N\to \N$ such that $\chi(G)\leq f(\omega(G))$ for every $G\in \CC$, where $\chi(G)$ and $\omega(G)$ respectively stand for the chromatic number and the clique number of $G$. Bonnet et al.\ proved that every class of graphs of bounded twin-width is $\chi$-bounded~\cite[Theorem~4]{tww3}, or more precisely, if a graph $G$ has twin-width at most $t\geq 3$, then $\chi(G)\leq (t+2)^{\omega(G)-1}$~\cite[Theorem~21]{tww3}. They asked whether the $\chi$-bounding function could be polynomial in the clique number; equivalently, whether classes of bounded twin-width are {\em{polynomially $\chi$-bounded}}. 

Note that Bonamy and Pilipczuk proved that classes of graphs of bounded cliquewidth are indeed polynomially $\chi$-bounded~\cite{BonamyP20}. The proof, however, heavily relies on the tree-like structure of graphs of bounded cliquewidth and it seems difficult to lift it to the setting of twin-width. On the other hand, Gajarsk\'y et al.\ proved that classes of graphs of bounded twin-width that additionally exclude a half-graph as a semi-induced subgraph are even {\em{linearly $\chi$-bounded}}, that is, the $\chi$-bounding function is linear in the clique number~\cite{GajarskyPT21}. However, as pointed out in~\cite{BonamyP20}, a construction of Chudnovsky et al.~\cite{ChudnovskyPST13} shows that already in the case of graphs of bounded cliquewidth, the degree of the polynomial bound on the $\chi$-bounding function has to grow with the cliquewidth. So this has to be the case for classes of bounded twin-width as well, provided they are indeed polynomially $\chi$-bounded in the first place.

In this work we provide a significant step towards resolving the question of Bonnet et al.~\cite{tww3} by showing that classes of graphs of bounded twin-width are {\em{quasi-polynomially $\chi$-bounded}}. More precisely, we prove the following statement.

\begin{theorem}\label{thm:main}
 For every $t\in \N$ there exists a constant $\gamma_t\in \N$ such that for every graph $G$ of twin-width at most $t$ and clique number $\omega$, we have
 $$\chi(G)\leq 2^{\gamma_t\cdot \log^{4t+3} \omega}.$$
\end{theorem}

Note that for every fixed $t$, the bound provided by Theorem~\ref{thm:main} is quasi-polynomial in $\omega$, but the degree of the polylogarithmic factor in the exponent depends on $t$. We do not know how to reduce the bound to a polynomial, or even to a quasi-polynomial of the form $2^{\gamma_t\cdot \log^c \omega}$ for a constant $c$ independent of $t$.

Let us briefly comment on the key conceptual differences between the proof of (exponential) $\chi$-boundedness of Bonnet et al.~\cite{tww3} and our proof of Theorem~\ref{thm:main}. The proof of Bonnet et al.\ applies a standard strategy in the area of $\chi$-boundedness. Namely, they show that one can partition the vertex set of a graph of twin-width $t$ into $t+2$ subsets so that each subset induces a subgraph with the clique number smaller by at least $1$. Then induction is applied to each of these subgraphs, with the clique number being the progress measure in the induction. Without modifications, this strategy inherently leads to an exponential bound on the $\chi$-bounding function. In our proof, we use two different induction steps:
\begin{itemize}
 \item In one step, we induct on induced subgraphs in which the clique number drops significantly: by a constant fraction.
 \item In the second step, we induct on graphs that may possibly have even larger clique number (but bounded polynomially in the original one), but in which an auxiliary progress measure --- the largest size of an almost mixed minor --- drops by at least one. This auxiliary measure is bounded in terms of the twin-width, so this step can be applied only a constant number of times, provided the twin-width is originally bounded by a constant.
\end{itemize}
The auxiliary progress measure used in the second step is expressed through the non-existence of certain structures in the adjacency matrix of the graph. For this reason, the entire reasoning needs to be conducted in the matrix setting. We believe that the idea behind the second induction step is somewhat novel in the context of twin-width and we hope that it may find applications beyond the scope of this work.

\section{Preliminaries}
\label{sec:preliminaries}

All logarithms throughout this paper are base $2$.
For a positive integer $s$, we write $[s]=\{1,\ldots,s\}$. We use standard graph notation. In particular, $\chi(G)$ and $\omega(G)$ respectively denote the chromatic number and the clique number of a graph $G$.

\paragraph*{Pure pairs.}
Let $G$ be a graph. A pair $A,B$ of disjoint subsets of vertices of $G$ is {\em{complete}} if every vertex of $A$ is adjacent to every vertex of $B$, and {\em{anti-complete}} if there is no edge with one endpoint in $A$ and the other in $B$. The pair $A,B$ is {\em{pure}} if it is complete or anti-complete, and {\em{impure}} otherwise. A vertex $v$ is complete or anti-complete {\em{towards}} a set $B$ with $v\notin B$ if the pair $\{v\},B$ is complete or anti-complete, respectively. Finally, we say that $A$ is {\em{semi-pure towards}} $B$ if every vertex of $A$ is either complete or anti-complete towards $B$.

\paragraph*{Twin-width.} We now recall the definition of the twin-width of a graph. We do it only for the sake of completeness, as later we will rely only on the known connections between the boundedness of twin-width and the non-existence of large mixed minors in adjacency matrices. In particular, essentially the whole reasoning behind the proof of Theorem~\ref{thm:main} is conducted in the setting of matrices.

A {\em{vertex partition}} of $G$ is a partition $\Pc$ of the vertex set of $G$. The operation of {\em{contracting}} two distinct parts $A,B$ of $\Pc$ produces a new partition $\Pc'$ obtained from $\Pc$ by replacing parts $A,B$ with a new part $A\cup B$. A {\em{contraction sequence}} for $G$ is a sequence $\Pc_1,\Pc_2,\ldots,\Pc_n$
of vertex partitions of $G$, where $n$ is the vertex count of $G$, such that
\begin{itemize}[nosep]
 \item $\Pc_1$ is the finest partition where every vertex is in its own part;
 \item $\Pc_n$ is the coarsest partition where all vertices are in the same part; and
 \item $\Pc_{i+1}$ is obtained from $\Pc_i$ by contracting two distinct parts of $\Pc_i$, for all $i\in [n-1]$.
\end{itemize}
The {\em{width}} of such contraction sequence is defined as the smallest integer $t$ such that for every $i\in [n]$ and every part $A$ of $\Pc_i$, there are at most $t$ other parts $B$ in $\Pc_i$ such that $A,B$ is an impure pair. The {\em{twin-width}} of $G$ is the least possible width of a contraction sequence of $G$.

\paragraph*{Matrices and mixed minors.}
As mentioned, in this work we mostly rely on an understanding of twin-width of graphs through mixed minors in their adjacency matrices. This understanding is the cornerstone of the theory of twin-width and was developed in~\cite{tww1}. We need several definitions concerning divisions of matrices.

\begin{definition}[\cite{tww1}]
A~matrix is {\em{horizontal}} if all its rows are constant, and {\em{vertical}} if all its columns is constant.
If a matrix is both horizontal and vertical, then we call it {\em{constant}}.
If a matrix is neither horizontal nor vertical, then we call it {\em{mixed}}.
Note that a~matrix $M$ is mixed if and only if it contains a~{\em{corner}}: a~$2 \times 2$ mixed contiguous submatrix of $M$.
\end{definition}

\begin{definition}[\cite{tww1}]
Let $M$ be a matrix. A subset of rows of $M$ is {\em{convex}} if it is contiguous in the order of rows in $M$; same for subsets of columns. 
A~{\em{division}} of $M$ is a pair $\Dc=(\Rc,\Cc)$, where $\Rc$ is a partition of rows into convex subsets, called {\em{row blocks}}, and $\Cc$ is a partition of columns into convex subsets, called {\em{column blocks}}. We call $\Dc$ an {\em{$s$-division}} if $|\Rc|=|\Cc|=s$, and in this case we index the row and column blocks of $\Dc$ with integers from $[s]$ according to the natural order of blocks.
The intersection of the $i$th row block and the $j$th column block ($i, j \in [s]$) is a~contiguous submatrix of $M$, which we call a {\em{zone}} of $\Dc$ and denote by $\Dc[i, j]$.
Similarly, by $\Dc[i_1 \dots i_2,\ j_1 \dots j_2]$ we mean the contiguous submatrix of $M$ consisting of the union of zones $\Dc[i,j]$ for $i_1\leq i\leq i_2$ and $j_1\leq j\leq j_2$. 
Supposing $M$ is symmetric, we call $\Dc$ {\em{symmetric}} if $\Rc$ and $\Cc$ partition rows and columns in exactly the same way; that is, projecting $\Rc$ and $\Cc$ naturally onto indices of rows and columns produces the same partition. 
\end{definition}

\begin{definition}
Let $M$ be a matrix.
A~$d$-division $\Dc$ of $M$ is a~\emph{$d$-mixed minor} if each zone of $\Dc$ is mixed.
We say that $M$ is \emph{$d$-mixed-free} if $M$ has no  $d$-mixed minor.
\end{definition}

The following result of Bonnet et al.~\cite{tww1} provides the connection between the twin-width of a graph and mixed minors in its adjacency matrix. Note here that to construct an adjacency matrix of a graph one has to specify a {\em{vertex ordering}}: a total order on the vertex set signifying the organization of rows and columns.  

\begin{theorem}[\cite{tww1}]\label{thm:tww-mm}
 Let $G$ be a graph and $t,d$ be positive integers. The following holds:
 \begin{itemize}
  \item If the twin-width of $G$ is at most $t$, then there is a vertex ordering of $G$ for which the adjacency matrix of $G$ is $(2t+2)$-mixed-free.
  \item If there is a vertex ordering of $G$ for which the adjacency matrix of $G$ is $d$-mixed-free, then the twin-width of $G$ is at most $2^{2^{\Oh{d}}}$.
 \end{itemize}
\end{theorem}

The back-bone of the proof of Theorem~\ref{thm:tww-mm} is the following lemma, whose proof is an easy application of the classic result of Marcus and Tardos~\cite{MarcusTardos}.

\begin{lemma}[\cite{tww1}]
\label{lem:marcus-tardos}
For every integer $d\geq 1$ there is a constant $c_d$ such that the following holds: if a matrix $M$ is $d$-mixed-free, then every $s$-division of $M$ contains at most $c_d \cdot s$ mixed zones.
\end{lemma}

In this work we will work with a slight relaxation of the notion of a mixed minor, explained next. On the technical level, this detail is somewhat analogous to the choice of working with the notion of {\em{quasi-index}} in~\cite{GajarskyPT21}.

\begin{definition}
A~$d$-division $\Dc$ of a matrix $M$ is a~\emph{$d$-almost mixed minor} if for each pair of indices $i, j \in [d]$ with $i \neq j$, the zone $\Dc[i, j]$ is mixed.
A~matrix is \emph{$d$-almost mixed-free} if it admits no $d$-almost mixed minors.
\end{definition}

Thus, almost mixed minors differ from mixed minors in that it is not required that the ``diagonal'' zones are mixed.
Clearly, every $d$-almost mixed free matrix is also $d$-mixed free. On the other hand, every $d$-mixed-free matrix is $2d$-almost mixed-free. To see that, note that if there was a $2d$-almost mixed minor, then its first $d$ row blocks and last $d$ column blocks would induce a submatrix with a $d$-mixed minor. Also, note that any submatrix of a $d$-(almost) mixed-free matrix is also $d$-(almost) mixed-free; we will often use this fact~implicitly.

\section{Obtaining the recurrence}\label{sec:recurrence-proof}

In this section we provide the main step towards the proof of Theorem~\ref{thm:main}, which is a recursive upper bound on the $\chi$-bounding function of graphs admitting a $d$-almost mixed-free adjacency matrix. After giving some preliminary tools in Section~\ref{sec:compressions}, we formulate this main step in Lemma~\ref{lem:main} in Section~\ref{sec:main-lemma-statement}, and devote the remainder of Section~\ref{sec:recurrence-proof} to the proof of this lemma.

\subsection{Compressions}\label{sec:compressions}

We start with some auxiliary results on {\em{horizontal}} and {\em{vertical compressions}} of matrices. These will be later used for handling recursive steps that apply to submatrices with smaller excluded almost mixed minors.

Call a matrix $M$ {\em{graphic}} if it is symmetric, has all entries in $\{0,1\}$, and all entries on the diagonal of $M$ are $0$. In other words, $M$ is graphic if it is the adjacency matrix of some graph (with respect to some vertex ordering).

\begin{definition}
  Let $\Dc$ be a symmetric $s$-division of a graphic matrix $M$. We define the \emph{horizontal compression} $G_\Dc^\Hs$ as the graph over vertex set $[s]$ where for any $1 \leq i < j \leq s$, we have $ij \in E(G_\Dc^\Hs)$ if and only if the zone $\Dc[i, j]$ is non-zero horizontal (equivalently, $\Dc[j, i]$ is non-zero vertical).
  We define the \emph{vertical compression} $G_\Dc^\Vs$ symmetrically. The \emph{mixed compression} $G_\Dc^\Ms$ is defined analogously, but we put an edge $ij$ whenever $i\neq j$ and the zone $\Dc[i,j]$ is mixed. 
\end{definition}

Let $M_\Dc^\Hs$, $M_\Dc^\Vs$, and $M_\Dc^\Ms$ be the adjacency matrices of $G_\Dc^\Hs$, $G_\Dc^\Vs$, and $G_\Dc^\Ms$ respectively, with respect to the natural vertex orderings inherited from $M$. 

We first note that thanks to the Marcus-Tardos Theorem, the mixed compression of a $d$-mixed-free matrix is always sparse, and hence colorable with few colors.

\begin{lemma}\label{lem:mixed-degenerate}
 For every $d\in \N$ there exists a constant $C_d$ such that the following holds.
 Let $M$ be a $d$-mixed-free graphic matrix and $\Dc$ be a symmetric division of $M$. Then $\chi(G^\Ms_\Dc)\leq C_d$.
\end{lemma}
\begin{proof}
  Let $k$ be the number of row blocks (equivalently, of column blocks) of $\Dc$.
  By Lemma~\ref{lem:marcus-tardos} and the fact that $d$-mixed-freeness is closed under taking submatrices, for every subset $S \subseteq [k]$, there are at most $\frac12 c_d |S|$ pairs $i,j\in S$, $i<j$, such that the zone $\Dc[i,j]$ is mixed. (The factor $\frac12$ comes from the fact that $M$ is symmetric, so each such pair  contributes with two mixed zones.)
  Thus, $|E(G^\Ms_\Dc[S])| \leq \frac12 c_d |S|$.
  Since $S$ was chosen arbitrarily, it follows that $G^\Ms_\Dc$ is $c_d$-degenerate, and hence $\chi(G^\Ms_\Dc) \leq c_d + 1$.
  So we may set $C_d \coloneqq  c_d + 1$.
\end{proof}

We next observe that almost mixed minors in $M_\Dc^\Hs$ lift to almost mixed minors in~$M$.

\begin{lemma}
  \label{lem:almost-minor-lifting}
  Let $M$ be a graphic matrix and $\Dc$ be a symmetric division of $M$. Suppose $M_\Dc^\Hs$ contains a~$d$-almost mixed minor $\Ec$ for some  $d \geq 2$. (Note that $\Ec$ is not necessarily symmetric.) Then $M$ contains a~$d$-almost mixed minor $\Lc$ such that:
  \begin{itemize}[nosep]
    \item $\Lc$ is a coarsening of $\Dc$ (that is, each row block of $\Lc$ is the~union of some convex subset of row blocks of $\Dc$, and the same applies for column blocks); and
    \item each zone of $\Lc$ is formed by an~intersection of at least two row blocks of $\Dc$ and at least two column blocks of $\Dc$.
  \end{itemize}
\end{lemma}
  \begin{proof}
    We lift $\Ec$ to a~$d$-almost mixed division $\Lc$ of $M$ naturally as follows.
    For $i\in [d]$, if the $i$th row block of $\Ec$ spans rows $r_1\dots r_2$ of $M_\Dc^\Hs$, then we set the $i$th row block of $\Lc$ to be the union of row blocks of $\Ec$ from the $r_1$th to the $r_2$th. Similarly for column blocks. Thus, if $i,j\in [d]$ and $\Ec[i, j]$ is the intersection of rows $r_1 \dots r_2$ and columns $c_1 \dots c_2$ of $M_\Dc^\Hs$, then    
    $$\Lc[i, j]=\Dc[r_1 \dots r_2,\, c_1 \dots c_2].$$
    Clearly, $\Lc$ is a coarsening of $\Dc$.
    Further, since $d \geq 2$ and $\Ec$ is a $d$-almost mixed minor of $M_\Dc^\Hs$, it is easy to see that each row block and each column block of $\Lc$ has to span at least two blocks of $\Dc$, for otherwise no corner in $M_\Dc^\Hs$ would fit into this block. 
    So it remains to show that each zone $\Lc[i, j]$ with $i,j\in [d]$, $i \neq j$, is mixed.
    
    Fix some $i, j \in [d]$ with $i \neq j$.
    Assume that $\Ec[i, j]$ is formed by the intersection of rows $r_1 \dots r_2$ and columns $c_1 \dots c_2$ of $M_\Dc^\Hs$.
    Since $\Ec[i, j]$ is mixed, there is a~corner $C=M_\Dc^\Hs[r\dots r+1,\, c\dots c+1]$ for some $r_1 \leq r < r_2$ and $c_1 \leq c < c_2$.
    It is then enough to show that the submatrix $A := \Dc[r \dots r+1,\, c \dots c+1]$ of $M$ is mixed.
    Indeed, since $A$ is a~submatrix of $\Lc[i, j] = \Dc[r_1 \dots r_2, c_1 \dots c_2]$, it will follow that $\Lc[i, j]$ is mixed as well.
    
    We perform a case study, depending on the value of $c - r$ (which intuitively signifies how close $C$ is to the diagonal of $M_\Dc^\Hs$):
    
    \smallskip
    \underline{Case 1a:} $c \geq r + 2$. Then $C$ is strictly above the diagonal of $M_\Dc^\Hs$.
    So for each $i \in \{r, r+1\}$ and $j \in \{c, c+1\}$, we have $M_\Dc^\Hs[i, j] = 1$ if and only if $\Dc[i, j]$ is non-zero horizontal.
    If $A$ were horizontal, then for each $i \in \{r, r+1\}$ we would have $M_\Dc^\Hs[i, c] = M_\Dc^\Hs[i, c+1]$, and $C$ would be horizontal; a contradiction with $C$ being a corner.
    Similarly, if $A$ were vertical, then for each $j\in \{c,c+1\}$ we would have $M_\Dc^\Hs[r, j] = M_\Dc^\Hs[r+1, j]$ and $C$ would be vertical; again a contradiction with $C$ being a corner.
    So $A$ must be mixed.
    
    
    \smallskip
    \underline{Case 1b:} $c \leq r - 2$. Then $C$ is strictly below the diagonal of $M_\Dc^\Hs$.
    So for each $i \in \{r, r+1\}$ and $j \in \{c, c+1\}$, we have $M_\Dc^\Hs[i, j] = 1$ if and only if $\Dc[i, j]$ is non-zero vertical.
    Hence, the proof is symmetric to Case 1a.
    
    \smallskip
    \underline{Case 2a:} $c = r + 1$. Then $C$ intersects the diagonal of $M_\Dc^\Hs$ at entry $M_\Dc^\Hs[r + 1, c]$, while the remaining entries are above the diagonal.
    Observe that $\Dc[r + 1, c]$ is graphic, and thus either constant $0$ or mixed.
    If it is mixed, then $A$ is already mixed as well. So from now on assume that $\Dc[r + 1, c]$ is constant $0$.
    Consequently, $M_\Dc^\Hs[r+1, c] = 0$.
    
    First, assume that $A$ is horizontal.
    Then, $\Dc[r + 1, c + 1]$ is constant $0$ as well and thus $M_\Dc^\Hs[r + 1, c + 1] = 0$.
    Moreover, $\Dc[r,\, c\dots c+1]$ is horizontal, implying that $M_\Dc^\Hs[r, c] = M_\Dc^\Hs[r, c+1]$. (Note that both these entries are equal to $1$ if and only if $\Dc[r, c]$ is non-zero, or equivalently if $\Dc[r, c+1]$ is non-zero.)
    So $C$ cannot be a~corner in $M_\Dc^\Hs$, a contradiction.
    
    Next, assume that $A$ is vertical.
    Analogously, $\Dc[r, c]$ is constant $0$, hence $M_\Dc^\Hs[r, c] = 0$, while $\Dc[r\dots r+1,\, c+1]$ is vertical, implying that $M_\Dc^\Hs[r, c+1] = M_\Dc^\Hs[r+1, c+1]$. So again, we conclude that $C$ cannot be a~corner, a contradiction.
    
    \smallskip
    \underline{Case 2b:} $c = r - 1$. This case is symmetric to Case 2a.
    
    \smallskip
    \underline{Case 3:} $c = r$. That is, $C$ intersects the diagonal of $M_\Dc^\Hs$ at $M_\Dc^\Hs[r, c]$ and $M_\Dc^\Hs[r + 1, c + 1]$, while $M_\Dc^\Hs[r+1, c]$ is below the diagonal and $M_\Dc^\Hs[r, c+1]$ is above the diagonal.
    As in Case 2a, each of $\Dc[r, c]$ and $\Dc[r + 1, c + 1]$ is either mixed or constant $0$. If any of them is mixed, then $A$ is mixed as well and we are done; so assume that that both $\Dc[r, c]$ and $\Dc[r + 1, c + 1]$ are constant $0$.
    Now, if $A$ were horizontal or vertical, then both $\Dc[r+1,c]$ and $\Dc[r,c+1]$ would be constant $0$ as well, implying that $C$ would be constant $0$. This is a contradiction with $C$ being a corner.
    
    \smallskip
    This finishes the proof.
  \end{proof}
Naturally, a statement symmetric to Lemma~\ref{lem:almost-minor-lifting} applies to $M_\Dc^\Vs$. Thus:

\begin{corollary}
  Suppose $M$ is a graphic matrix that is $d$-almost mixed free for some $d\geq 2$, and $\Dc$ is a symmetric division of $M$. Then both $M_\Dc^\Hs$ and $M_\Dc^\Vs$ are $d$-almost mixed free.
\end{corollary}

The next lemma is the main outcome of this section. It shows that if $M$ is the adjacency matrix of a graph $G$ (with respect to some vertex ordering) and $M$ has no large almost mixed minor, then the clique number in the compressions of $M$ are controlled in terms of the clique number of $G$.

\begin{lemma}
  \label{lem:clique_lemma}
  Let $G$ be a graph and denote $\omega=\clique{G}$. Let $M$ be the adjacency matrix of $G$ in some vertex ordering and let $\Dc$ be a symmetric division of $M$. Suppose $M$ has no $d$-almost mixed minor that is a coarsening of $\Dc$, for some $d\geq 1$.
  Then $$\clique{G_\Dc^\Hs} \leq 2 \binom{\omega + d - 2}{d - 1} \leq 2\omega^{d-1}.$$
\end{lemma}
  \begin{proof}
  Let 
  $$\mu(\omega,d)=2\binom{\omega + d - 2}{d - 1} - 1.$$
  It can be easily verified that $\mu(\omega,d)$ satisfies the following recursive definition:
    \begin{itemize}[nosep]
     \item $\mu(1, \cdot) = \mu(\cdot, 1) = 1$; and
\item $\mu(\omega, d) = \mu(\omega - 1, d) + \mu(\omega, d - 1) + 1$ for $\omega, d \geq 2$.
    \end{itemize}
    We now show that $\clique{G_\Dc^\Hs} \leq \mu(\omega, d)$ by induction on $\omega$ and $d$.
    
    If $\omega = 1$ or $d=1$, then $G$ must be edgeless. So $G_\Dc^\Hs$ is edgeless as well, implying that $\clique{G_\Dc^\Hs} \leq 1$. This resolves the base of the induction, so from now on assume that $d,\omega\geq 2$.
    
    Assume for contradiction that $\clique{G_\Dc^\Hs} > \mu(\omega, d)$.
    By restricting $G$ to a suitable induced subgraph,
    without loss of generality, we may assume that $G_\Dc^\Hs$ is a~complete graph with $k \coloneqq \mu(\omega, d) + 1$ vertices. Note that this means that for each $1\leq i<j\leq k$, the zone $\Dc[i,j]$ is non-zero horizontal.
    Since we assume that $\Dc$ is a~symmetric division, $\Dc$ yields a~partition of $V(G)$ into $k$ subsets $V_1, V_2, \dots, V_k$. These subsets are convex in the vertex ordering used in the construction of $M$, and are indexed naturally according to this vertex ordering.
    
    Let $\ell\coloneqq \mu(\omega, d-1) + 1$ and 
    $$Y \coloneqq V_1 \cup V_2 \cup \dots \cup V_{\ell}\qquad\textrm{and}\qquad Z \coloneqq V(G) \setminus Y.$$
    Note that $k=\mu(\omega, d) + 1=\mu(\omega,d-1)+\mu(\omega-1,d)+2$, hence $Y$ is the union of $\ell=\mu(\omega,d-1)+1$ parts $V_i$ and $Z$ is the union of $\mu(\omega-1,d)+1$ parts $V_i$.
    By induction assumption applied to the graph $G[Z]$ with division $\Dc$ restricted to blocks corresponding to the vertices of $Z$, we infer that
    $$\omega(G[Z]) > \omega - 1.$$
    This implies that there is no vertex $v \in Y$ that would be complete towards $Z$ (i.e., adjacent to every vertex of $Z$).
    Indeed, if this were the case, then $v$ together with the largest clique in $G[Z]$ would form a clique in $G$ of size at least $\omega+1$, a contradiction.
    
    Since no vertex of $Y$ is complete towards $Z$, for every $i \in [\ell]$ one can find $j(i) > \ell$ such that the zone $\Dc[i, j(i)]$ in $M$ is \emph{not} constant $1$.
    However, as $M_\Dc^\Hs[i, j(i)] = 1$, we know that $\Dc[i, j(i)]$ is non-zero horizontal.
    Hence, $\Dc[i, j(i)]$ is non-constant horizontal.
    
    For every $i \in [\ell]$, consider the submatrix $$R_i \coloneqq \Dc[i,\, \ell+1 \dots k].$$
    Note that $R_i$ contains the zone $\Dc[i,j(i)]$.
    This means that $R_i$ cannot be vertical, as its contiguous submatrix $\Dc[i, j(i)]$ is non-constant horizontal.
    It also cannot be horizontal, as then there would exist a vertex $v \in V_i\subseteq Y$ that would be complete towards $Z$.
    Therefore, $R_i$ is mixed.
    By symmetry, the submatrix $S_i \coloneqq  \Dc[\ell+1 \dots k,\, i]$ is mixed as well (Figure \ref{sfig:compression-mixed-strips}).
    
    \begin{figure}
    \centering
    \begin{subfigure}[b]{0.45\textwidth}
    \centering
    \begin{tikzpicture}[scale=0.9]

\foreach \x/\nx in {0.0/0.4, 0.4/0.9, 0.9/1.3, 1.3/1.7, 1.7/2.0, 2.0/2.3, 2.3/2.6, 2.6/3.0} {
  \draw[red, fill=red!80!white, draw opacity=0.5, fill opacity=0.2] (3.0, -\x) rectangle (6.5, -\nx);
  \draw[red, fill=red!80!white, draw opacity=0.5, fill opacity=0.2] (\x, -3.0) rectangle (\nx, -6.5);
}

\draw[draw=black, very thick] (0, 0) rectangle (6.5, -6.5);

\foreach \x in {0.4, 0.9, 1.3, 1.7, 2.0, 2.3, 2.6, 3.0, 3.4, 3.8, 4.4, 4.8, 5.2, 5.5, 5.9, 6.2} {
  \draw[black, dotted] (\x, 0) -- (\x, -6.5);
  \draw[black, dotted] (0, -\x) -- (6.5, -\x);
}
\draw[black] (3.0, 0) -- (3.0, -6.5);
\draw[black] (0, -3.0) -- (6.5, -3.0);

\node at (1.5, 0.5) {\small $Y$};
\node at (-0.5, -1.5) {\small $Y$};
\node at (4.75, 0.5) {\small $Z$};
\node at (-0.5, -4.75) {\small $Z$};

\end{tikzpicture}
    \caption{}
    \label{sfig:compression-mixed-strips}
    \end{subfigure}
    \hfill
    \begin{subfigure}[b]{0.45\textwidth}
    \centering
    \begin{tikzpicture}[scale=0.9]

\draw[draw=black, very thick] (0, 0) rectangle (6.5, -6.5);

\foreach \x in {0.4, 0.9, 1.3, 1.7, 2.0, 2.3, 2.6, 3.0, 3.4, 3.8, 4.4, 4.8, 5.2, 5.5, 5.9, 6.2} {
  \draw[black, dotted] (\x, 0) -- (\x, -6.5);
  \draw[black, dotted] (0, -\x) -- (6.5, -\x);
}

\node at (1.5, 0.5) {\small $Y$};
\node at (-0.5, -1.5) {\small $Y$};
\node at (4.75, 0.5) {\small $Z$};
\node at (-0.5, -4.75) {\small $Z$};

\tikzset{extline/.style={red!70!black, very thick, dashed, dash pattern=on 3pt off 3pt}};

\foreach \xa/\xb/\ya/\yb in {0/1.3/0.9/1.7, 0/1.3/1.7/3.0, 1.3/2.3/0/0.9, 1.3/2.3/1.7/3.0, 2.3/3.0/0/0.9, 2.3/3.0/0.9/1.7} {
  \draw[fill=red!70!black, fill opacity=0.3, draw=none] (\xa,-\ya) rectangle (\xb,-\yb);
}
\foreach \xa/\xb/\ya/\yb in {3.0/6.5/0/0.9, 3.0/6.5/0.9/1.7, 3.0/6.5/1.7/3.0, 0/1.3/3.0/6.5, 1.3/2.3/3.0/6.5, 2.3/3.0/3.0/6.5} {
\draw[fill=red!80!white, fill opacity=0.2, draw=none] (\xa,-\ya) rectangle (\xb,-\yb);
}

\foreach \x in {1.3, 2.3, 3.0} {
  \draw[red!70!black, thick] (\x, 0) -- (\x, -3.0);
  \draw[extline] (\x, -3.0) -- (\x, -6.5);
}
\foreach \x in {0.9, 1.7, 3.0} {
  \draw[red!70!black, thick] (0, -\x) -- (3.0, -\x);
  \draw[extline] (3.0, -\x) -- (6.5, -\x);
}
\draw[extline] (3.0, 0) -- (3.0, -6.5);
\draw[extline] (0, -3.0) -- (6.5, -3.0);

\end{tikzpicture}
    \caption{}
    \label{sfig:compression-almost-mixed-minor}
    \end{subfigure}
    \caption{Setup in the proof of Lemma~\ref{lem:clique_lemma}. \\
    (a) A~symmetric division $\Dc$ of matrix $M$ (dotted). Each red horizontal strip ($R_i$) and each red vertical strip $(S_i$) is mixed. \\
    (b) Each $(d-1)$-almost mixed minor of $M'$ that is a~coarsening of $\Dc$ (dark red) can be extended to a~$d$-almost mixed minor of $M$ by adding $Z$ as the~final row and column block.
    }
    \label{fig:compression-setup}
    \end{figure}
    
    Let $$M' \coloneqq \Dc[1 \dots \ell,\ 1 \dots \ell]$$ be the adjacency matrix of $G[Y]$ in the vertex ordering inherited from $G$.
    We observe that $M'$ cannot contain any $(d-1)$-almost mixed minor $\Ec$ that would be a coarsening of $\Dc$ (restricted to blocks corresponding to the vertices of $Y$). Indeed, otherwise we could form a~$d$-almost mixed minor $\Ec'$ of $M$ by adding to $\Ec$ one column block, spanning the suffix of columns of $M$ corresponding to the vertices of $Z$, and one row block, spanning the suffix of rows of $M$ corresponding to the vertices of $Z$ (Figure \ref{sfig:compression-almost-mixed-minor}).
    By the mixedness of $R_i$ and $S_i$ for each $i \in [\ell]$, we see that for each $j \in [d-1]$, the zones $\Ec'[d, j]$ and $\Ec'[j, d]$ would be mixed. So $\Ec'$ would be a $d$-almost mixed minor in $M$ that would be coarsening of $\Dc$, a contradiction. 
    
    As $\ell=\mu(\omega,d-1)+1$,
    we may now apply the induction assumption to the graph $G[Y]$ with division $\Dc$ restricted to blocks corresponding to the vertices of $Y$. We thus infer that
    $$\clique{G[Y]} > \omega.$$
    As $G[Y]$ is an induced subgraph of $G$, this is a contradiction that finishes the proof.
\end{proof}

We remark that a symmetric reasoning shows that the same conclusion applies also to~$G_\Dc^\Vs$: under the assumptions of Lemma~\ref{lem:clique_lemma}, we also have $\clique{G_\Dc^\Vs} \leq 2\omega^{d-1}$.

\subsection{Statement of the main lemma}\label{sec:main-lemma-statement}

With auxiliary tools prepared, we can proceed to the main result of this section. Let $f_d\colon \Z_{>0}\to \Z_{>0}$ be defined as follows: for $\omega\in \N$, $f_d(\omega)$ is the maximum chromatic number among graphs that admit a vertex ordering yielding a $d$-almost mixed free adjacency matrix.
Note that by~Theorem~\ref{thm:tww-mm} and~\cite[Theorem~4]{tww3}, $f_d(\omega)$ is finite for all $\omega\in \N$. We have the following easy observation.

\begin{lemma}\label{lem:tww-f}
 Every graph of twin-width at most $t$ and clique number $\omega$ has chromatic number at most~$f_{4t+4}(\omega)$. 
\end{lemma}
\begin{proof}
 Every graph of twin-width at most $t$ has a vertex ordering that yields a $(2t+2)$-mixed-free adjacency matrix~(Theorem~\ref{thm:tww-mm}), and every $(2t+2)$-mixed-free matrix is also $(4t+4)$-almost mixed-free. Given this, the claim follows from the definition of $f_{4t+4}(\omega)$.
\end{proof}

For convenience, let us extend the domain of $f_d$ to $\R_{>0}$ by setting $f_d(x)=f_d(\lceil x\rceil)$ for every positive non-integer $x$.
The main step towards a better bound on $f_d$ is the recurrence provided by the following lemma.

\begin{lemma}\label{lem:main}
 Let $d,\omega,k$ be integers satisfying $d\geq 3$, $\omega \geq 5$, and $1\leq k<\omega/4$. Then there exists a constant $C_d$ depending only on $d$ such that
 \begin{equation}
  \label{eq:main}
  f_d(\omega) \leq f_d(\omega - k) + C_d \left[ f_d(\omega -k) + 8C_df_{d-1}(2\omega^{d-1})^2 \cdot \sum_{u=0}^{\left\lfloor \log_2 k\right\rfloor} f_d(2^{u + 1}) \cdot f_d\left(\frac{2k}{2^u} + 1\right) \right].
\end{equation}
\end{lemma}

The remainder of this section is devoted to the proof of Lemma~\ref{lem:main}. For this, fix $d,\omega,k$ as in the premise of Lemma~\ref{lem:main}, as well as a graph $G$ that, under some vertex ordering $\leq$, admits a $d$-almost mixed-free adjacency matrix $M$. Our goal is to construct a proper coloring of $G$ with the number of colors bounded by the right hand side of~\eqref{eq:main}. The construction is obtained by a sequence of {\em{coloring steps}}, each of which constructs a part of the coloring and reduces the remaining task to a simpler one.

For simplicity we may assume that $V(G) = [n]$ and $\leq$ is the standard order on $[n]$, so that vertices are equal to the indices of their rows and columns in $M$.

%
%
%
%

\subsection{Forming blobs and simplifying connections}

The first step of the construction is to partition the vertex set of $G$ into parts, called {\em{blobs}}, which are significantly simpler in terms of the clique number.
Formally, we construct blobs $B_1,B_2,\ldots,B_{m+1}$ by an inductive procedure as follows. Supposing $B_1,\ldots,B_{i-1}$ are already defined for some $i\geq 1$, we let $B_{i}$ be the smallest prefix of $V(G)\setminus \bigcup_{j=1}^{i-1} B_{j}$ in the ordering $\leq$ that satisfies $\clique{G[B_i]}\geq \omega-k$. If no such prefix exists, we finish the construction by setting $m\coloneqq i-1$ and $B_{m+1}\coloneqq V(G)\setminus \bigcup_{j=1}^{i-1} B_{j}$. Since adding one vertex can increase the clique number by at most $1$, it is easy to see that the blobs satisfy the following assertions:
\begin{itemize}[nosep]
 \item Blobs $B_1,\ldots,B_{m+1}$ form a partition of $V(G)$.
 \item  Each blob is {\em{convex}} in the ordering $\leq$, that is, its vertices form an interval in $\leq$. Moreover, the blobs are ordered by $\leq$ in according to their indices: if $1\leq i<j\leq m+1$, then $u<v$ for all $u\in B_i$ and $v\in B_j$.
 \item For each $i\in [m]$ we have $\clique{G[B_i]} = \omega - k$.
 \item We have $\clique{G[B_{m+1}]} \leq \omega - k$.
\end{itemize}

The first step in constructing a coloring of $G$ is to resolve the part $G[B_{m+1}]$.

\begin{coloring}
  Color $G[B_{m+1}]$ using a~separate palette of $f_d(\omega - k)$ colors; such a proper coloring exists due to $\clique{G[B_{m+1}]}\leq \omega-k$.
  From now on, we may disregard $B_{m+1}$ from further considerations.
  That is, our goal is to properly color $G[B_1 \cup B_2 \cup \dots \cup B_m]$ using
  \[
  C_d \left[ f_d(\omega -k) + 8C_df_{d-1}(2\omega^{d-1})^2 \cdot \sum_{u=0}^{\left\lfloor \log_2 k\right\rfloor} f_d(2^{u + 1}) \cdot f_d\left(\frac{2k}{2^u} + 1\right)\right]\ \ \text{colors}.
  \]
\end{coloring}

For a subset of rows $X$ and a subset of columns $Y$, by $M[X,Y]$ we denote the matrix obtained from $M$ by deleting all rows not belonging to $X$ and all columns not belonging to~$Y$. We sometimes call $M[X,Y]$ the {\em{connection}} between $X$ and $Y$.
The next goal is to resolve mixed connections between the blobs. This is easy thanks to Lemma~\ref{lem:mixed-degenerate}.

\begin{lemma}
  \label{lem:marcus-tardos-degeneracy}
Let $C_d$ be the constant provided by Lemma~\ref{lem:mixed-degenerate}. Then there exists a~partition $[m] = A_1 \cup A_2 \cup \dots \cup A_{C_d}$ into (not necessarily convex) subsets so that for any given $i \in [C_d]$ and a~pair $b, c \in A_i$ with $b \neq c$, the matrix $M[B_b,B_c]$ is \emph{not} mixed.
\end{lemma}
\begin{proof}
 It suffices to apply Lemma~\ref{lem:mixed-degenerate} to the mixed compression of the matrix $M[\bigcup_{i=1}^m B_i,\bigcup_{i=1}^m B_i]$ along its division into blobs $B_1,\dots,B_m$.
\end{proof}

\begin{coloring}\label{step:mixed}
  Let $[m] = A_1 \cup A_2 \cup \dots \cup A_{C_d}$ be the partition provided by Lemma~\ref{lem:marcus-tardos-degeneracy}. Assign a separate palette of colors to each set $A_i$, $i\in [C_d]$. That is, supposing we properly color each graph $G[\bigcup_{a\in A_i} B_a]$ using
  \begin{equation}
  \label{eq:after-mixed-coloring}
  f_d(\omega -k) + 8C_df_{d-1}(2\omega^{d-1})^2 \cdot \sum_{u=0}^{\left\lfloor \log_2 k\right\rfloor} f_d(2^{u + 1}) \cdot f_d\left(\frac{2k}{2^u} + 1\right)\ \ \text{colors,}
  \end{equation}
  we may construct a proper coloring of $G[B_1\cup \dots \cup B_m]$ by taking the union of the colorings of $G[\bigcup_{a\in A_i} B_a]$, $i\in [m]$, on disjoint palettes.
\end{coloring}

For the simplicity of presentation, from now on we focus on a single set $A_i$. That is, by restricting attention to the induced subgraph $G[\bigcup_{a\in A_i} B_a]$ with vertex ordering inherited from $G$, we may assume that $V(G)=B_1\cup \dots \cup B_m$ and there are no mixed connections between any pair of different blobs.
Then our goal is to properly color $G$ using the number of colors given by \eqref{eq:after-mixed-coloring}.

Let $\Dc$ be the symmetric $m$-division of $M$ given by the partitioning into blobs.
Thus, $\Dc$ has no mixed zones outside of the main diagonal, so for $i,j\in [m]$, $i \neq j$, the connection $\Dc[i,j]=M[B_i,B_j]$ is of one of the following types:
\begin{itemize}
  \item \emph{Empty} if $\Dc[i, j]$ is constant $0$.
  \item \emph{Non-constant horizontal} if $\Dc[i,j]$ is horizontal but not constant. In graph-theoretic terms, this means that $B_i$ is semi-pure towards $B_j$. (Recall that this means that each vertex $v \in B_i$ is either complete or anti-complete towards $B_j$.)
  \item \emph{Non-constant vertical} if $\Dc[i,j]$ is vertical but not constant. Again, in graph-theoretic terms this means that $B_j$ is semi-pure towards $B_j$.
\end{itemize}
Here, note that the connection $\Dc[i,j]$ cannot be constant $1$. This is because then the pair $B_i,B_j$ would be complete, implying that
$$\clique{G}\geq \clique{G[B_i]}+\clique{G[B_j]}\geq 2(\omega-k)>\omega,$$
a contradiction.

Given a vertex $v$, say $v\in B_i$ for some $i\in [m]$, we say that $v$ is \emph{rich} if it is complete towards any other blob $B_j$, $j\neq i$. Otherwise, we say that $v$ is \emph{poor}. We next observe that poor vertices can be only adjacent within single blobs.

\begin{lemma}\label{lem:poor-nonadjacent}
  Suppose $u$ and $v$ are poor vertices such that $u\in B_i$ and $v\in B_j$ for some $i\neq j$. Then $u$ and $v$ are non-adjacent.
\end{lemma}
  \begin{proof}
    Assume otherwise.
    Then $M[u,v] = M[v,u] = 1$.
    Thus, the zone $\Dc[i, j]$ is not constant~$0$ due to containing $M[u,v]$.
    If it was horizontal, then $u$ would be pure towards the entire blob $B_j$, so $u$ would be rich.
    Symmetrically, if $\Dc[i,j]$ was vertical, then $v$ would be rich. 
    Since both $u$ and $v$ are assumed to be poor, we have a~contradiction.
  \end{proof}

Let $Z$ be the set of all poor vertices. By Lemma~\ref{lem:poor-nonadjacent}, we see that $G[Z]$ is the~disjoint union of graphs $G[Z \cap B_1], \dots, G[Z \cap B_m]$. It follows that $\clique{G[Z]}=\omega - k$.

\begin{coloring}
  Color $Z$ using a~separate palette of $f_d(\omega - k)$ colors; this can be done due to $\clique{G[Z]}=\omega - k$.
  It remains to properly color the graph $G - Z$ using
  \[
  8C_df_{d-1}(2\omega^{d-1})^2 \cdot \sum_{u=0}^{\left\lfloor \log_2 k\right\rfloor} f_d(2^{u + 1}) \cdot f_d\left(\frac{2k}{2^u} + 1\right)\ \ \text{colors}.
  \]
\end{coloring}

For each $i\in [m]$, let $B'_i \coloneqq B_i \setminus Z$. Intuitively, our next goals are to first construct a~proper coloring of each subgraph $G[B'_i]$ separately, and then use the coloring of the vertices within the blobs to resolve the semi-pure interblob connections.

\newcommand{\fst}{\mathsf{first}}
\newcommand{\lst}{\mathsf{last}}
\newcommand{\lft}{\mathsf{L}}
\newcommand{\rgt}{\mathsf{R}}

\subsection{Forming and analyzing subblobs}

Fix $i \in [m]$ and consider the set $B'_i$ with the ordering inherited from $G$.
We now want to find a~proper coloring of $G[B'_i]$.
We partition $B'_i$ into four (not necessarily convex) subsets $B'_{i,x,y}$ with $x, y \in \{0, 1\}$.
We put each $v \in B'_i$ into $B'_{i,x,y}$ for $(x,y)$ defined as follows:
\begin{itemize}
  \item $x=M[v,\fst]$, where $\fst=1$ is the first vertex of $G$ in the ordering $\leq$; and
  \item $y=M[v,\lst]$, where $\lst=|V(G)|$ is the last vertex of $G$ in the ordering $\leq$.
\end{itemize}

Now, fix $x, y \in \{0, 1\}$ for a moment.
We partition $B'_{i,x,y}$ into sets $I_{i,x,y}^1 \cup I_{i,x,y}^2 \cup \dots \cup I_{i,x,y}^t$, called \emph{subblobs}, by induction as follows.
Assuming $I_{i,x,y}^1,\dots,I_{i,x,y}^{j-1}$ are already defined, $I_{i,x,y}^j$ is the largest prefix of $\leq$ restricted to $B'_{i,x,y}\setminus \bigcup_{s=1}^{j-1} I_{i,x,y}^s$ with the following property: all vertices of $I_{i,x,y}^j$ are {\em{twins}} with respect to $V(G)\setminus B_i$ (that is, in $G$ they have exactly the same neighborhood in $V(G)\setminus B_i$). The construction finishes when every vertex of $B'_{i,x,y}$ is placed in a subblob. Note that subblobs $I_{i,x,y}^j$ are not necessarily convex in $\leq$, but they are convex in $\leq$ restricted to $B'_{i,x,y}$, and they are ordered naturally by $\leq$: $w<w'$ for all $w\in I_{i,x,y}^{j}$ and $w'\in I_{i,x,y}^{j'}$ with $j<j'$ (Figure~\ref{fig:subblob-example}).
Note that we require that the vertices within every subblob are twins with respect to all blobs different than $B_i$ in the entire graph $G$ (where $G$ contains both rich and poor vertices).

We observe that every subblob induces a graph of small clique number.

\begin{lemma}\label{lem:subblob-clique}
  For each $j \in [t]$, we have $\clique{G[I_{i,x,y}^j]} \leq k$.
\end{lemma}
  \begin{proof}
    Since $I_{i,x,y}^j$ consists of rich twins with respect to $G - B_i$, it follows that there exists some blob $B_{i'}$, $i'\neq i$, such that the pair $I_{i,x,y}^j,B_{i'}$ is complete. But we have $\clique{G[B_{i'}]} = \omega - k$, so it follows that $\clique{G[I_{i,x,y}^j]}\leq \omega-(\omega-k)=k$.
  \end{proof}
  
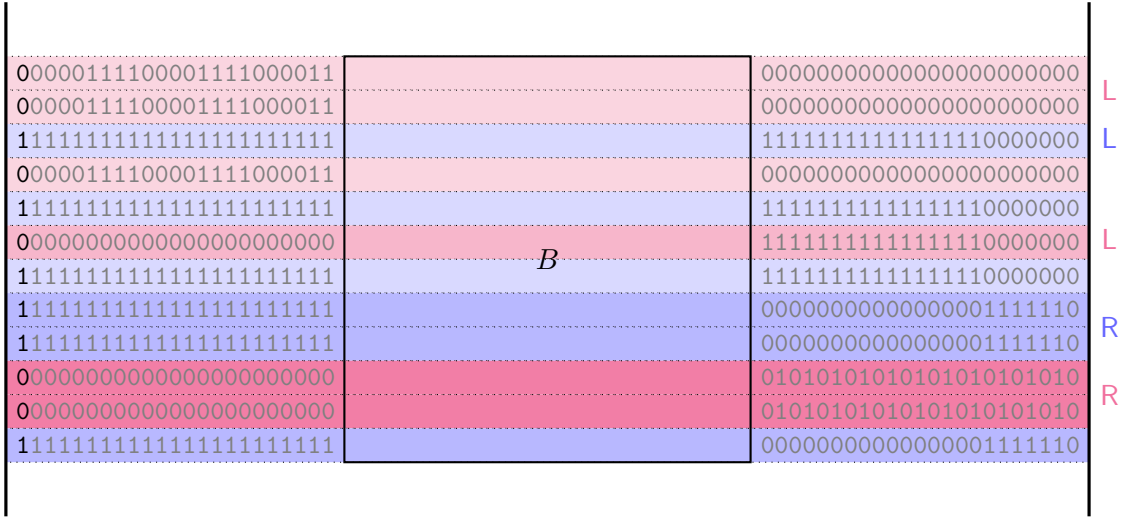
\begin{figure}
  \centering
  \begin{tikzpicture}[scale=0.9]

\newcommand{\colorrx}{RubineRed!20!white}
\newcommand{\colorry}{RubineRed!35!white}
\newcommand{\colorrz}{RubineRed!64!white}
\newcommand{\colorbx}{blue!15!white}
\newcommand{\colorby}{blue!28!white}

\foreach \y/\color in {0/\colorrx, -0.5/\colorrx, -1.0/\colorbx, -1.5/\colorrx, -2.0/\colorbx, -2.5/\colorry, -3.0/\colorbx, -3.5/\colorby, -4.0/\colorby, -4.5/\colorrz, -5.0/\colorrz, -5.5/\colorby} {
  \draw[dotted, fill=\color] rectangle (0, \y) rectangle ($(16, \y) - (0, 0.5)$);
}

\foreach \y/\a/\b/\c in {
    0.0/0/0000111100001111000011/00000000000000000000000,
    -0.5/0/0000111100001111000011/00000000000000000000000,
    -1.0/1/1111111111111111111111/11111111111111110000000,
    -1.5/0/0000111100001111000011/00000000000000000000000,
    -2.0/1/1111111111111111111111/11111111111111110000000,
    -2.5/0/0000000000000000000000/11111111111111110000000,
    -3.0/1/1111111111111111111111/11111111111111110000000,
    -3.5/1/1111111111111111111111/00000000000000001111110,
    -4.0/1/1111111111111111111111/00000000000000001111110,
    -4.5/0/0000000000000000000000/01010101010101010101010,
    -5.0/0/0000000000000000000000/01010101010101010101010,
    -5.5/1/1111111111111111111111/00000000000000001111110
    } {
  \node [label=below right:{\texttt{\small \a\textcolor{gray}{\b}}}] at ($(0, \y)+(-0.15, 0.18)$) {};
  \node [label=below right:{\texttt{\small \textcolor{gray}{\c}}}] at ($(11, \y)+(-0.15, 0.18)$) {};
}

\node at (16.3, -0.5) {\textcolor{RubineRed!70!white}{$\lft$}};
\node at (16.3, -1.2) {\textcolor{blue!60!white}{$\lft$}};
\node at (16.3, -2.7) {\textcolor{RubineRed!70!white}{$\lft$}};
\node at (16.3, -4.0) {\textcolor{blue!60!white}{$\rgt$}};
\node at (16.3, -5.0) {\textcolor{RubineRed!70!white}{$\rgt$}};

\draw[black, thick] (5, 0) rectangle (11, -6);
\node at (8, -3) {$B$};

\draw[black, very thick] (0, 0.8) -- (0, -6.8);
\draw[black, very thick] (16, 0.8) -- (16, -6.8);

\let\colorrx\undefined
\let\colorry\undefined
\let\colorrz\undefined
\let\colorbx\undefined
\let\colorby\undefined

\end{tikzpicture}
  \caption{A~blob $B$ with $12$ rich vertices and its partitioning into the subblobs.
  For simplicity, each vertex of $B$ has $y = 0$ (i.e., no vertices of $B$ are connected to $\lst$). \\
  The subblobs containing vertices with $x = 0$ are marked with different shades of red, and the subblobs containing vertices with $x = 1$ are marked with different shades of blue.
  For each subblob, its type $z \in \{\lft, \rgt\}$ is marked on the right of the matrix.}
  \label{fig:subblob-example}
\end{figure}

We now divide the subblobs into a logarithmic number of {\em{buckets}} according to the clique number of the subgraphs induced by them. We remark that this bucketing approach is inspired by the proof of Chudnovsky et al.~\cite{ChudnovskyPST13} that polynomial $\chi$-boundedness is preserved under closure by the substitution operation. Intuitively, the idea is to capture the tradeoff between the heaviness of the subblobs in terms of their clique number, and the sparseness of the graph of connections between the subblobs. These two quantities are respectively represented by the terms $f_d(2^{u + 1})$ and $f_d\left(\frac{2k}{2^u} + 1\right)$ in the right hand side of~\eqref{eq:main}, and clearly these two terms ``play against each other''.

Let $\ell \coloneqq \left\lfloor \log k \right\rfloor$.
We partition the subblobs $I_{i,x,y}^1, \dots, I_{i,x,y}^t$ into $2(\ell + 1)$ {\em{buckets}}
$\Sc_{i,x,y,z, u}$ for $z \in \{\lft, \rgt\}$ and $u \in \{0, 1, \dots, \ell\}$ using the following process.
For every $j \in [t]$, let $v_j$ be any vertex of $I_{i,x,y}^j$.
We put $I_{i,x,y}^j$ into the bucket $\Sc_{i,x,y,z, u}$ for $(z,u)$ defined as follows:
\begin{itemize}
  \item If $j = 1$ then $z = \lft$.
    Otherwise, that is for $j \geq 2$, we know that there exists a~vertex $a \in B_{i'}$ with $i' \neq i$ such that $M[v_j, a] \neq M[v_{j-1}, a]$.
    Pick any such vertex.
    If $i' < i$, then we put $z = \lft$; otherwise, put $z = \rgt$ (Figure \ref{fig:subblob-example}).
  \item $u$ is such that $2^u \leq \clique{G[I_{i,x,y}^j]} < 2^{u+1}$.
\end{itemize}
Observe that Lemma~\ref{lem:subblob-clique} ensures that every subblob is placed in a bucket.

If a~subblob $I_{i,x,y}^j \subseteq B'_{i,x,y}$ belongs to the bucket $\Sc_{i,x,y,z, u}$, we call $I_{i,x,y}^j$ an~\emph{$(x,y,z,u)$-subblob}. Let $W_{x,y,z,u}$ be the union of all $(x,y,z,u)$-subblobs in $G$, that is,
$$W_{x,y,z,u}=\bigcup_{i=1}^m \bigcup \Sc_{i,x,y,z, u}.$$
The idea is to assign a separate palette to every choice of $(x,y,z,u)$ as above.

\begin{coloring}\label{col:focus-xyzu}
  For each quadruple of parameters $(x, y, z, u)\in \{0,1\}^2\times \{\lft,\rgt\}\times \{0, 1, \dots, \ell\}$, we assign a~separate palette for coloring the subgraph $G[W_{x,y,z,u}]$ with
  \begin{equation}\label{eq:col-xyzu}
    C_d \cdot f_{d-1}(2\omega^{d-1})^2  \cdot f_d(2^{u+1}) \cdot f_d\left(\frac{2k}{2^u} + 1\right)\ \ \text{colors}.
  \end{equation}
  That is provided we properly color every subgraph $G[W_{x,y,z,u}]$ with that many colors, we can color the whole graph $G$ using the union of those coloring on separate palettes.
\end{coloring}

Therefore, from now on we fix a quadruple $(x, y, z, u)\in \{0,1\}^2\times \{\lft,\rgt\}\times \{0, 1, \dots, \ell\}$ and focus on coloring $G[W_{x,y,z,u}]$ using as many colors as specified in \eqref{eq:col-xyzu}. Denote $W\coloneqq W_{x,y,z,u}$ for brevity.

First, using we construct a coloring that at least deals with edges within subblobs.

\begin{coloring}\label{col:intrasubblob}
  For each $(x,y,z,u)$-subblob $I$, properly color the subgraph $G[I]$ using $f_d(2^{u+1})$ colors; this is possible due to $\clique{G[I]}\leq 2^{u+1}$. Take the union of these colorings using {\em{the same}} palette of $f_d(2^{u+1})$ colors. This is a coloring of $W$ with the property that every two adjacent vertices of $W$ belonging to the same subblob receive different colors. Call this coloring~$\lambda_1$.
\end{coloring}

Coloring $\lambda_1$ defined above already properly colors all the edges within subblobs. Our next goal is to refine $\lambda_1$ to a coloring that also properly color edges connecting vertices from different subblobs. These come in two different types: the subblobs may be either contained in the same blob, or be contained in different blobs. Consequently, the refinement is done in two steps corresponding to the two types.

Let us fix $i\in [m]$ and enumerate the bucket $\Sc_{i,x,y,z,u}$ as $\{I_{i,x,y}^{j(1)}, I_{i,x,y}^{j(2)}, \dots, I_{i,x,y}^{j(\alpha)}\}$, where $\alpha\coloneqq |\Sc_{i,x,y,z,u}|$ and $j(1) < j(2) < \dots < j(\alpha)$.
Recall that for each subblob, we have previously chosen an arbitrary vertex $v_{j(b)} \in I_{i,x,y}^{j(b)}$. We observe the following.

\begin{lemma}
  \label{lem:finding-mixed-matrices-outside-blobs}
  For every $c \in \{2, 3, \dots, \alpha\}$, we have the following.
  \begin{itemize}
    \item If $z = \lft$, then the submatrix $M[\{w\colon v_{j(c-1)}\leq w\leq v_{j(c)}\}, \bigcup_{i'<i} B_{i'} ]$ is mixed;
    \item If $z = \rgt$, then the submatrix $M[\{w\colon v_{j(c-1)}\leq w\leq v_{j(c)}\},\bigcup_{i'>i} B_{i'}]$ is mixed.
  \end{itemize}
\end{lemma}
  \begin{proof}
    Assume that $z = \lft$.
    The proof for $z = \rgt$ is symmetric, so we omit it.
    
    Since $I_{i,x,y}^{j(c)} \in \Sc_{i,x,y,z, u}$, the vertices $v_{j(c)}$ and $v_{j(c) - 1}$ have different neighborhoods in the set $\bigcup_{i'<i} B_{i'}$. That is, there is a vertex $a\in \bigcup_{i'<i} B_{i'}$ such that $M[v_{j(c) - 1}, a] \neq M[v_{j(c)}, a]$. Note that the existence of $a$ implies that $i>1$, so in particular $\fst\notin B_i$.
    By the construction, we have  $M[v_{j(c) - 1}, \fst] = M[v_{j(c)}, \fst] = x$. It follows that $M[\{v_{j(c - 1)}, \dots ,v_{j(c)}\}, \bigcup_{i'<i} B_{i'} ]$ contains a mixed $2\times 2$ submatrix, so it is mixed as~well.
  \end{proof}

Let $G_{i}\coloneqq G[\bigcup \Sc_{i,x,y,z, u}]$, and let $M_{i}\coloneqq M[\bigcup \Sc_{i,x,y,z, u},\bigcup \Sc_{i,x,y,z, u}]$ be its adjacency matrix in the order inherited from $G$.
Let also $\Dc_{i}$ be the (symmetric) $\alpha$-division of $M_{i}$ according to the boundaries of subblobs in $\Sc_{i,x,y,z, u}$.

Naturally, for each $p < q$, if the zone $\Dc_{i}[p, q]$ is non-zero, then it is of at least one of the following types: mixed (type~$\Ms$), non-zero horizontal (type $\Hs$) or non-zero vertical (type $\Vs$). Our goal is to construct three colorings of the subblobs in $\Sc_{i}$ that respectively take care of these three types of connections; these will be called $\phi^\Ms$, $\phi^\Hs$, and $\phi^\Vs$, respectively. 
Hence, it is natural to define $G^\Hs_{i}$ to be the horizontal compression of $M_{i}$ along its division $\Dc_{i}$, and similarly let $G^\Vs_{i}$ and  $G^\Ms_{i}$ be the corresponding vertical and mixed compressions. So $\phi^\Ms$, $\phi^\Hs$, and $\phi^\Vs$ should be just proper colorings of $G^\Ms_{i}$, $G^\Hs_{i}$, and $G^\Vs_{i}$, respectively.

Obtaining $\phi^\Ms$ is easy.
Namely, by Lemma~\ref{lem:mixed-degenerate}, the graph $G^\Ms_{i}$ admits a proper coloring $\phi^\Ms$ with $C_d$ colors; here, $C_d$ is the constant provided by Lemma~\ref{lem:mixed-degenerate}.

We now show how to obtain colorings $\phi^\Hs$ and $\phi^\Vs$.  
Let $M^\Hs_{i}$ and $M^\Vs_{i}$ be the adjacency matrices of $G^\Hs_{i}$ and $G^\Vs_{i}$, respectively, in the natural order inherited from $G$.
The next lemma is the key conceptual step:
we observe that the complexity of the matrices $M^\Hs_{i}$ and $M^\Vs_{i}$ has dropped.

\begin{lemma}
  $M^\Hs_{i}$ is $(d - 1)$-almost mixed-free.
\end{lemma}
  \begin{proof}
    We only prove the lemma for $z = \lft$. For $z = \rgt$ the proof is analogous, so we omit it.
    
    Aiming towards a contradiction, suppose that $M^\Hs_{i}$ contains a~$(d-1)$-almost mixed minor~$\Ec$.
    Since $d \geq 3$, Lemma~\ref{lem:almost-minor-lifting} applies, and there exists a~$(d-1)$-almost mixed minor $\Lc$ of $M_{i}$ which is a~coarsening of $\Dc_{i}$, and each (row or column) block of $\Lc$ spans at least two (row or column) blocks of $\Dc_{i}$.
    We now construct a~$d$-almost mixed minor $\Lc'$ of $M$ in the following way:
    \begin{itemize}
      \item the first row block of $\Lc'$ spans rows of $\bigcup_{i'<i} B_{i'}$ of $M$;
      \item the first column block of $\Lc'$ spans columns of $\bigcup_{i'<i} B_{i'}$ of $M$;
      \item the $i$th row block of $\Lc'$ ($i \in \{2, 3, \dots, d\}$) spans all rows in $M$ that are spanned by the $(i-1)$st row block of $\Lc$ in $M_{i}$; analogously for the $i$th column block.
    \end{itemize}
  There is a technical detail here: as defined above, formally $\Lc'$ is a division of a submatrix of $M$ induced by rows and columns of $\bigcup_{i'<i} B_{i'}\cup V(G_{i})$. This can be easily fixed by expanding row and column blocks of $\Lc'$ in any convex way so that they cover all of rows and columns of $M$. This way the zones can only get larger.

  It remains to show that $\Lc'$ is indeed a $d$-almost mixed minor of $M$.
  Naturally, each $\Lc'[p, q]$ for $p, q \geq 2$, $p \neq q$, is mixed due to the mixedness of $\Lc[p-1, q-1]$.
  Also, $\Lc'[p, 1]$ is mixed for $p \geq 2$ for the following reason: the $(p-1)$st row block of $\Lc$ contains the rows corresponding to the vertices $v_{j(c-1)}, v_{j(c)}$ for some $c \in \{2, 3, \dots, \alpha\}$, and then,
  by Lemma~\ref{lem:finding-mixed-matrices-outside-blobs}, the submatrix $M[\{w\colon v_{j(c-1)}\leq w\leq v_{j(c)}\},\bigcup_{i'<i} B_{i'}]$ is mixed.
  This submatrix is also a~submatrix of the zone $\Lc'[p, 1]$, so this zone is mixed as well.
  That $\Lc'[1, q]$ is mixed for every $q \geq 2$ follows from a~symmetric argument.
  Thus, $\Lc'$ is indeed a~$d$-almost mixed minor of $M$; a~contradiction.
  \end{proof}

A symmetric proof shows that $M^\Vs_{i}$ is also $(d-1)$-almost mixed free.
By Lemma~\ref{lem:clique_lemma}, we have $\clique{G^\Hs_{i}}\leq 2\omega^{d-1}$, so by we conclude that $G^\Hs_{i}$ admits a proper coloring $\phi^\Hs$ using $f_{d-1}(2\omega^{d-1})$ colors. By a symmetric reasoning,  $G^\Vs_{i}$ also admits a proper coloring $\phi^\Vs$ using $f_{d-1}(2\omega^{d-1})$ colors.

\begin{coloring}
  For every $i\in [m]$, 
  construct a coloring $\lambda^i_2$ of $V(G_{i})$ as follows:
  for every subblob $I\in \Sc_{i, x, y, z, u}$ and $v\in I$, we let
  $$\lambda^i_2(v)=(\lambda_1(v),\phi^\Ms(I),\phi^\Hs(I),\phi^\Vs(I)),$$
  where $\phi^\Ms$, $\phi^\Hs$, and $\phi^\Vs$ are constructed as above.
  Let $\lambda_2$ be the union of colorings $\lambda^i_2$ for $i\in [m]$ using the same palette of $C_d\cdot f_d(2^{u+1})\cdot f_{d-1}(2\omega^{d-1})^2$ colors. Thus, $\lambda_2$ is a coloring of $W$ using $C_d\cdot f_d(2^{u+1})\cdot f_{d-1}(2\omega^{d-1})^2$ colors that satisfies the following property: for every pair of adjacent vertices $w,w'\in W$ that belong to the same blob,  we have $\lambda_2(w)\neq \lambda_2(w')$. 
\end{coloring} 

Denote $\Lambda\coloneqq C_d\cdot f_d(2^{u+1})\cdot f_{d-1}(2\omega^{d-1})^2$.
Let $F^t$ for $t\in [\Lambda]$ be the color classes of $\lambda_2$. Clearly, each $F^t$  is a subset of $W$ such that $F^t\cap B_i$ is an independent set for each $i\in [m]$. We observe that sets $F^t$ induce subgraphs with relatively small clique numbers.

\begin{lemma}\label{lem:Ft}
 For each $t\in [\Lambda]$, we have 
  $\clique{G[F^t]} \leq 2\left\lfloor \frac{k}{2^u} \right\rfloor + 1$.
\end{lemma}
  \begin{proof}
    Denote $\beta \coloneqq \left\lfloor \frac{k}{2^u} \right\rfloor$ for brevity.
    Suppose that $G[F^t]$ contains a~clique $K$ of size $2\beta + 2$.
    As $F^t$ intersects every blob on an independent set, every vertex of $K$ comes from a~different blob. For $v\in K$, let $B(v)$ be the blob containing $v$.
    
    Construct a digraph $T$ on vertex set $K$ as follows: for distinct $v,v'\in K$, add an arc $(v,v')$ to $T$ if $B(v)$ is semi-pure towards $B(v')$. Since the division $\Dc$ induced by blobs is assumed to have no mixed zones (as a result of Coloring Step~\ref{step:mixed}), $T$ is semi-complete, that is, for all distinct $v,v'\in K$ at least one of the arcs $(v,v')$ and $(v',v)$ is present. It follows that there exists $w\in K$ whose indegree in $T$ is at least $\beta+1$. In other words, there are vertices $v_1,\ldots,v_{\beta+1}\in K$, different from $w$, such that $B(v_i)$ is semi-pure towards $B(w)$ for each $i\in [\beta+1]$. In particular, this implies that each    
    $v_i$ is pure towards $B(w)$.
    
    Let $R$ be a maximum clique in the blob $B(w)$; recall that $|R|=\omega-k$. Further, for each $i\in [\beta+1]$, let $S_i$ be a maximum clique in the subblob of $B(v_i)$ containing $v_i$; recall that $2^u \leq |S_i| < 2^{u+1}$.
    Finally, let $$L\coloneqq R\cup S_1\cup \dots \cup S_{\beta+1}.$$
    Observe that 
    \[ |L| \geq 2^u \cdot \left(\left\lfloor \frac{k}{2^u} \right\rfloor + 1\right) + (\omega - k) > \omega. \]
    
    So to reach a contradiction, it remains to show that $S$ is a clique in $G$.
    To this end, pick two different vertices $a, b \in S$.
    If $a$ and $b$ come from the same blob, then, by the construction, they are adjacent.
    Otherwise, we consider two cases:
    \begin{itemize}
      \item $a \in S_p$ and $b \in S_q$ for two different $p, q \in [\beta + 1]$.
        Since $a$ and $v_p$ are in the same $(x, y, z, u)$-subblob, we get that $a$ and $v_p$ are twins with respect to the blob $B(v_q)$. In particular, $a$ and $v_p$ are both adjacent or both non-adjacent to $b$, or equivalently $M[a, b] = M[v_p, b]$.
        Similarly, $b$ and $v_q$ are twins with respect to the blob $B(v_p)$, implying $M[v_p, b] = M[v_p, v_q]$. But $v_p$ and $v_q$ are adjacent due to belonging to the clique $K$, so $M[a,b]=M[v_p,v_q]=1$. Hence $a$ and $b$ are adjacent as well.
      
      \item $a \in S_p$ and $b \in R$ for some $p \in [\beta + 1]$.
        As above, we have $M[a, b] = M[v_p, b]$.
        But since the blob $B(v_p)$ is semi-pure towards $B(w)$, and $b\in B(w)$, we have $M[v_p, b] = M[v_p, w]$. But $v_p$ and $w$ are adjacent due to belonging to the clique $K$, so $M[a,b]=M[v_p,w] = 1$. Hence again, $a$ and $b$ are adjacent.\qedhere
    \end{itemize}
  \end{proof}

We can now finalize the construction as follows.
  
\begin{coloring}
  For each $t\in [\Lambda]$, properly color $G[F^t]$ using a separate palette of
  \[
    f_d\left(2\left\lfloor \frac{k}{2^u}\right\rfloor + 1\right) \leq
    f_d\left(\frac{2k}{2^u} + 1\right)\ \ \text{colors}.
  \]
  This is possible by Lemma~\ref{lem:Ft}.
\end{coloring}

Thus, we have obtained a proper coloring of $G[W]=G[W_{x,y,z,u}]$ using $\Lambda\cdot f_d\left(\frac{2k}{2^u} + 1\right)$ colors, which fulfills the task set out in Coloring Step~\ref{col:focus-xyzu}. So this concludes the proof of~Lemma~\ref{lem:main}.
%
%
%
%

\newcommand{\RHS}{\mathsf{RHS}}
\newcommand{\eps}{\varepsilon}

\section{Solving the recurrence}

Our goal in this section is to prove the following statement, which will be later combined with Lemma~\ref{lem:main}.

\begin{lemma}\label{lem:recursion}
 Suppose $f_2,f_3,\ldots$ are functions from $\Z_{>0}$ to $\Z_{>0}$ satisfying the following:
 \begin{itemize}[nosep]
  \item $f_2(n)\leq n$ for all integers $n\geq 1$;
  \item $f_d(1)=1$ for all integers $d\geq 2$; and 
  \item for every integer $d\geq 3$ there exists $\alpha_d\in \N$ such that for every integer $n\geq 8$,
  \begin{equation}\label{eq:bobr}
  f_d(n)\leq \alpha_d\left[f_d(\lceil 7n/8\rceil)+f_{d-1}(2n^{d-1})^2\cdot \sum_{u=0}^{\lfloor \log n\rfloor-3} f_d(2^{u+1})\cdot f_d\left(\left\lceil\frac{n}{2^{u+1}}\right\rceil+1\right)\right].
  \end{equation}
 \end{itemize}
 Then for every integer $d\geq 2$ there exists a constant $\beta_d\in \N$ such that
 \begin{equation}\label{eq:wydra}
f_d(n)\leq 2^{\beta_d \log^{d-1} n} \qquad\textrm{for all }n\in \Z_{>0}.
 \end{equation}
\end{lemma}
\begin{proof}
 We proceed by induction on $d$. Clearly, for $d=2$ we may choose $\beta_1=1$.
 
 In the induction step for $d\geq 3$, we choose $\beta_d$ to be large enough so that the following inequalities are satisfied:
 \begin{gather}
  2^{\beta_d} \geq f_d(i)\quad \textrm{for all }i\in \{1,\ldots,2^{16}-1\},\label{eq:bs}\\
  \beta_d\cdot \frac{d-1}{2^{d+1}}\geq 3\beta_{d-1}d^{d-2},\label{eq:bd-}\\
  2^{\beta_d}\geq 2\alpha_d.\label{eq:ad}
 \end{gather}
 We now verify \eqref{eq:wydra} by induction on $n$. The base cases $n\in \{1,2,\ldots,2^{16}-1\}$ hold trivially thanks to~\eqref{eq:bs}, so from now on let us assume that $n\geq 2^{16}$.
 
 Let
 $$g_d(x)\coloneqq 2^{\beta_d \log^{d-1} x}$$
 and denote the right-hand side of~\eqref{eq:bobr} by $\RHS$. Using both induction assumptions, we have
 \begin{equation}\label{eq:nutria}
  \RHS\leq \alpha_d\left[g_d(\lceil 7n/8\rceil)+g_{d-1}(2n^{d-1})^2\cdot \sum_{u=0}^{\lfloor \log n\rfloor-3} g_d(2^{u+1})\cdot g_d\left(\left\lceil\frac{n}{2^{u+1}}\right\rceil+1\right)\right].
  \end{equation}
 By the convexity of function $x\mapsto x^{d-1}$, we have
 $$g_d(a)\cdot g_d(b)\leq g_d(2)\cdot g_d(ab/2)=2^{\beta_d}\cdot g_d(ab/2)\qquad \textrm{for all }a,b\geq 2.$$
 Therefore, for $u\in \{0,1,\ldots,\lfloor \log n\rfloor-3\}$,
 $$g_d(2^{u+1})\cdot g_d\left(\left\lceil\frac{n}{2^{u+1}}\right\rceil+1\right)\leq 2^{\beta_d}\cdot g_d\left (2^{u}\left\lceil\frac{n}{2^{u+1}}\right\rceil+2^{u}\right)\leq 2^{\beta_d}\cdot g_d\left(n/2+2^{u+1}\right)\leq 2^{\beta_d}\cdot g_d(3n/4).$$
 By combining this with~\eqref{eq:nutria} and observing that for $n\geq 2^{16}$ we have $\lceil 7n/8\rceil\leq 15n/16$, we conclude that
 \begin{equation}\label{eq:pizmak}
\RHS \leq \alpha_d\left[g_d(15n/16)+g_{d-1}(2n^{d-1})^2\cdot n\cdot 2^{\beta_d}\cdot g_d(3n/4)\right].
 \end{equation}
 
 In the estimation of the right hand side of~\eqref{eq:pizmak}
 we will need the following simple claim.
 
 \begin{claim}\label{cl:rev-bernoulli}
  For all reals $\eps\in [0,1/2]$ we have
  $$(1-\eps)^{d-1}\leq 1-\frac{d-1}{2^{d-2}}\cdot \eps.$$
 \end{claim}
 \begin{proof}
 Let $h(t)=(1+t)^{d-1}-\frac{d-1}{2^{d-2}}\cdot t-1$. Observe that for $t\in [-1/2,0]$, we have 
 $$h'(t)=(d-1)(1+t)^{d-2}-\frac{d-1}{2^{d-2}}\geq \frac{d-1}{2^{d-2}}-\frac{d-1}{2^{d-2}}=0.$$
 Since $h(0)=0$, it follows that $h(t)\leq 0$ for $t\in [-1/2,0]$; this is equivalent to the claim.
 \cqed\end{proof}
 
 First, let
 $$\nu\coloneqq g_d(15n/16).$$
 Observe that, by Claim~\ref{cl:rev-bernoulli},
 \begin{eqnarray}
  \log \nu & = & \beta_d (\log n - \log 16/15)^{d-1} = \beta_d\log^{d-1} n \left(1-\frac{\log 16/15}{\log n}\right)^{d-1}\label{eq:mors} \\
  & \leq & \beta_d\log^{d-1} n\left(1-\frac{(d-1)\log 16/15}{2^{d-2}}\cdot \frac{1}{\log n}\right) \leq \beta_d\log^{d-1} n - \frac{(d-1)\beta_d}{2^{d+2}}\log^{d-2} n.\nonumber
 \end{eqnarray}
 
 Next, let
 $$\mu\coloneqq g_{d-1}(2n^{d-1})^2\cdot n\cdot 2^{\beta_d}\cdot g_d(3n/4).$$
 Observe that
 \begin{eqnarray*}
  \log \mu & = & 2\beta_{d-1} \log^{d-2}(2n^{d-1})+\log n + \beta_d + \beta_d \log^{d-1} (3n/4)\\
  & \leq & 2\beta_{d-1}d^{d-2} \log^{d-2} n+\log n+\beta_d + \beta_d (\log n - \log 4/3)^{d-1}\\
  & \leq & 3\beta_{d-1}d^{d-2} \log^{d-2} n + \beta_d + \beta_d (\log n - \log 4/3)^{d-1}.
 \end{eqnarray*}
 By Claim~\ref{cl:rev-bernoulli}, we have
 \begin{eqnarray*}
 (\log n - \log 4/3)^{d-1} & = &\log^{d-1} n\cdot \left(1-\frac{\log 4/3}{\log n}\right)^{d-1} \\
 &\leq & \log^{d-1} n\cdot \left(1-\frac{(d-1)\log 4/3}{2^{d-2}}\cdot \frac{1}{\log n}\right)\leq \log^{d-1} n - \frac{d-1}{2^{d}}\cdot \log^{d-2} n.
 \end{eqnarray*}
 Therefore,
 \begin{eqnarray*}
  \log \mu & \leq & \beta_d \log^{d-1} n +\beta_d + \log^{d-2} n \cdot \left(3\beta_{d-1}d^{d-2} - \beta_d\cdot \frac{d-1}{2^{d}}\right)\\
  &\leq& \beta_d \log^{d-1} n + \beta_d -\frac{(d-1)\beta_d}{2^{d+1}}\log^{d-2} n, 
 \end{eqnarray*}
 where in the last inequality we used~\eqref{eq:bd-}. Further, since $n\geq 2^{16}$ and $d\geq 3$, we have
 $$\frac{d-1}{2^{d+1}}\log^{d-2} n\geq \frac{(d-1)\cdot 16^{d-2}}{2^{d+1}}\geq \frac{2\cdot 2^{4d-8}}{2^{d+1}}=2^{3d-8}\geq 2,$$
 hence
 \begin{equation}\label{eq:foka}
 \log \mu \leq \beta_d\log^{d-1} n -\frac{(d-1)\beta_d}{2^{d+2}}\log^{d-2} n.
 \end{equation}
 
 We now combine~\eqref{eq:mors} and~\eqref{eq:foka} with~\eqref{eq:pizmak}, thus obtaining:
 \begin{eqnarray*}
  \RHS & \leq & \alpha_d\left[2^{\log \nu}+2^{\log \mu}\right]\\  & \leq & \alpha_d\cdot 2^{\beta_d \log^{d-1} n}\cdot \left[2^{-\frac{(d-1)\beta_d}{2^{d+2}}\log^{d-2} n}+2^{-\frac{(d-1)\beta_d}{2^{d+2}}\log^{d-2} n}\right]\\
  & = & 2^{\beta_d\log^{d-1} n}\cdot \frac{2\alpha_d}{2^{\frac{(d-1)\beta_d}{2^{d+2}}\log^{d-2} n}}\leq 2^{\beta_d\log^{d-1} n}\cdot \frac{2\alpha_d}{2^{\frac{(d-1)\beta_d}{2^{d+2}}16^{d-2}}}\\
  & \leq & 2^{\beta_d\log^{d-1} n}\cdot \frac{2\alpha_d}{2^{\beta_d\cdot 2^{3d-9}}} \leq 2^{\beta_d\log^{d-1} n}\cdot \frac{2\alpha_d}{2^{\beta_d}} \leq 2^{\beta_d\log^{d-1} n}=g_d(n),
 \end{eqnarray*}
 where the last inequality follows from~\eqref{eq:ad}. As $f_d(n)\leq \RHS$, this concludes the proof.
\end{proof}

\section{Wrapping up the proof}

In this section we combine Lemmas~\ref{lem:main} and~\ref{lem:recursion} to obtain a proof of Theorem~\ref{thm:main}. However, since the statement of Lemma~\ref{lem:main} assumes $d\geq 3$, we need to consider the base case $d=2$ separately. This is provided by the following lemma. Recall here that a {\em{cograph}} is a $P_4$-free graph, that is, a graph that does not contain the path on $4$ vertices as an induced subgraph.

\begin{lemma}\label{lem:2cograph}
 Let $G$ be a graph that admits, under some vertex ordering, a $2$-almost mixed-free adjacency matrix. Then $G$ is a cograph.
\end{lemma}
\begin{proof}
 By contraposition and the fact that $2$-almost mixed-freeness is preserved under taking submatrices, it suffices to show that no vertex ordering of a $P_4$ yields a $2$-almost mixed-free adjacency matrix. Observe that if one partitions the vertex set of a $P_4$ into two parts of size $2$, then regardless of the choice of the partition, no part will be semi-pure towards the other. Therefore, for every vertex ordering of a $P_4$, dividing the corresponding adjacency matrix into four $2\times 2$ matrices yields a $2$-almost mixed minor.
\end{proof}

It is well known that cographs are {\em{perfect}}, that is, $\chi(G)=\omega(G)$ whenever $G$ is a cograph. Hence, from Lemma~\ref{lem:2cograph} we conclude that
$$f_2(\omega)\leq \omega\qquad\textrm{for all }\omega\in \Z_{>0}.$$
Further, we clearly have $f_d(1)=1$ for every $d\geq 2$, since graphs of clique number $1$ are edgeless. Finally, applying Lemma~\ref{lem:main} for $k=\lfloor\omega/8\rfloor$ yields that for all $\omega\geq 8$, $f_d(\omega)$ is upper bounded by 
\begin{eqnarray*}
 & & f_d(\lceil7\omega/8\rceil) + C_d \left[ f_d(\lceil7\omega/8\rceil) + 8C_df_{d-1}(2\omega^{d-1})^2 \cdot \sum_{u=0}^{\left\lfloor \log_2 \omega/8\right\rfloor} f_d(2^{u + 1}) \cdot f_d\left(\left\lceil\frac{ \omega}{2^{u+2}}\right\rceil + 1\right) \right]\\
 & \leq & 8C_d(C_d+1)\cdot \left[f_d(\lceil 7\omega/8\rceil)+f_{d-1}(2\omega^{d-1})^2\cdot \sum_{u=0}^{\lfloor \log \omega\rfloor-3} f_d(2^{u+1})\cdot f_d\left(\left\lceil\frac{\omega}{2^{u+1}}\right\rceil+1\right)\right].
\end{eqnarray*}
We may now apply Lemma~\ref{lem:recursion} to functions $f_2,f_3,\ldots$ to conclude the following.

\begin{theorem}\label{thm:main-mm}
 For every integer $d\geq 2$ there is a constant $\beta_d\in \N$ such that \mbox{$f_d(\omega)\leq 2^{\beta_d\cdot \log^{d-1} \omega}$} for every $\omega\in \Z_{>0}$. In other words, for every graph $G$ that has clique number $\omega$ and admits a $d$-almost mixed-free adjacency matrix under some vertex ordering, we have $\chi(G)\leq 2^{\beta_d\cdot \log^{d-1} \omega}$.
\end{theorem}

Theorem~\ref{thm:main} now follows from combining Theorem~\ref{thm:main-mm} with Lemma~\ref{lem:tww-f}.

\paragraph*{Acknowledgements.} We are grateful to Jakub Gajarsk\'y and Colin Geniet for many insightful discussions about the results presented in this article.

\bibliographystyle{plain}
\bibliography{references}

\end{document}